\documentclass[10pt]{article}
\usepackage{amssymb}
\usepackage{amsthm}
\usepackage{amsmath}
\usepackage{xcolor}
\usepackage{cancel}
\usepackage{enumerate}  

\newtheorem{theorem}{Theorem}[section]
\newtheorem{proposition}[theorem]{Proposition}

\newtheorem{lemma}[theorem]{Lemma}
\newtheorem{corollary}[theorem]{Corollary}

\newtheorem{question}[theorem]{Question}
\newcommand{\F}{\mathbb{F}}
\newcommand{\Z}{\mathbb{Z}}

\begin{document}
\title{A few more Hadamard\\
	Partitioned Difference Families}

\author{Anamari Naki\'c \thanks{Faculty of Electrical Engineering and Computing, University of Zagreb, Croatia, email: anamari.nakic@fer.hr}}

\maketitle

\begin{abstract}
A $(G,[k_1,\dots,k_t],\lambda)$ {\it partitioned difference family} (PDF) is a partition $\cal B$ of an additive group $G$ into sets 
({\it blocks}) of sizes $k_1$, \dots, $k_t$, such that the list of differences of ${\cal B}$ covers exactly $\lambda$ times every 
non-zero element of $G$. It is called {\it Hadamard} (HPDF) if the order of $G$ is $2\lambda$. 
The study of HPDFs is motivated by the fact that each of them gives rise, recursively,
to infinitely many other PDFs. Apart from the {\it elementary} HPDFs consisting of a Hadamard difference set and its complement, only
one HPDF was known.
In this article we present three new examples in several groups and we start a general investigation on the possible existence of HPDFs
with assigned parameters by means of simple arguments.
\end{abstract}


\section{Introduction}

We recall that  the list of differences of a subset $B$ of an additive group $G$, denoted by $\Delta B$, is the multiset
of all possible differences between two distinct elements of $B$:
$$
\Delta B: = \{ x-y \, : \, x \neq y, x,y \in B \}.
$$
More generally, the list of differences of a collection ${\cal F}=\{B_1, \ldots, B_t \}$ of subsets of $G$ 
is the multiset $\Delta{\cal F}$ which is union of the lists of differences of all the $B_i$'s, i.e., $\Delta{\cal F}=\bigcup_{i=1}^t\Delta B_i$.
The collection $\mathcal{F}$ is said to be a {\it difference family} (DF) if $\Delta{\cal F}$ covers 
every non-zero element of $G$ a constant number $\lambda$ of times. 
In this case, if $K$ is the multiset of the sizes of the $B_i$'s, one says that $\cal F$ is a $(G,K,\lambda)$-DF. 
More specifically, one writes $(G,[k_1,\dots,k_t],\lambda)$-DF if $B_i$ has size $k_i$. 
The $B_i$'s are the {\it base blocks} of $\cal F$ and $\lambda$ is its {\it index}.
One often speaks of a $(v,K,\lambda)$-DF or $(v,[k_1,\dots,k_t],\lambda)$-DF when the group $G$ (of order $v$) is understood.

If a DF of index $\lambda$ has only one block $B$ and its size is $k$, then one says that $B$ is a $(G,k,\lambda)$ {\it difference set} (DS).
As before, one generally speaks of a $(v,k,\lambda)$-DS if the group $G$ is understood.
Among the many classes of difference sets we have, in particular, the so called {\it Hadamard difference sets} 
which are those whose parameters are $(4u^2, 2u^2 - u, u^2 - u)$ for some $u$.

A $(G,[k_1,\dots,k_t],\lambda)$-DF is {\it  partitioned} (PDF) if its blocks partition $G$. 
Here the parameters $v=o(G)$ and $\lambda$ are completely determined by $K$.
Indeed it is evident that the equalities $k_1+\dots+ k_t=v$ and $k_1(k_1-1)+...+k_t(k_t-1)=\lambda(v-1)$ hold.
The notion of a PDF was introduced in \cite{DY05} for the construction of {\it optimal constant composition codes} \cite{LFHC03,DY05} 
which are important in various applications (see, e.g,. \cite{CCD04,GFM06})
and it is equivalent to that of a {\it zero difference balanced function} 
\cite{D} even though, according to \cite{BJ}, it is far preferable to keep the standard notation and terminology.

In a very recent article \cite{BJ2}, PDFs have been used to prove the existence of infinitely many {\it resolvable linear spaces}
with a mandatory set of block sizes and an automorphism group acting sharply transitively on the points and 
transitively on the parallel classes.

If $B$ is $(v,k,\lambda)$-DS in $G$, then $\{B,G\setminus B\}$ is a $(v,[k,v-k],v-2k+2\lambda)$-PDF.
So, in particular, a Hadamard $(4u^2, 2u^2 - u, u^2 - u)$-DS gives rise to a
$(4u^2, [2u^2+u, 2u^2 - u], 2u^2 )$-PDF.  Note that here the order of the group is twice the index.  
For this reason all PDFs with this property have been called {\it Hadamard partition difference families} (HPDFs) in 
the paper were they have been introduced \cite{Bur19}.
The motivation of their study is that each HPDF in a group $G$ gives rise to an infinite class of PDFs in 
suitable supergroups of $G$. 

As pointed out in \cite{BJ}, PDFs with (many) blocks of size $1$ are not considered very interesting. 
However, in the case of HPDFs, the situation is quite different since the PDFs arising from them by 
using Buratti's construction no longer contain blocks of size 1. 

Apart from the HPDFs arising from Hadamard DSs that we call {\it elementary}, there was only one known 
example of a HPDF. This makes believe that HPDFs are quite rare. In this paper we
determine three new examples in several groups.

Apart from this,
the article wants to be the beginning of a general investigation on HPDFs by means of quite simple arguments. It is structured as follows.

In the next section we characterize the PDFs having exactly two blocks. They
necessarily consist of a difference set and its complement. In particular, each HPDF with two blocks
necessarily consists of a Hadamard difference set and its complement.

In Section 3 we give the trivial necessary conditions for the existence of a HPDF, we list the first admissible
parameter sets and then we give some informations on the HPDFs with exactly three blocks. 

In Section 4, exploiting {\it partial difference sets}, we prove that a HPDF with three block sizes one of
which is 1 cannot exist.

In Section 5 we give one more necessary condition for the existence of a HPDF in a group
having at least one subgroup of index 2.

In Section 6 we present the three new examples mentioned earlier. We have a $(24,[1^3,2^2,17],12)$-HPDF in three
pairwise non-isomorphic groups, a $(36,[3,9,24],18)$-HPDF in nine pairwise non-isomorphic groups, and 
a cyclic $(40,[1,3,9,27],20)$-HPDF.

In Section 7 we determine the infinite families of {\it descendants} of our new three examples.

Finally, in Section 8 we pose two open questions. We first ask whether there exists a 
$(v,k,\lambda)$-DS with $v=2\lambda$ and $k>4$. The motivation is that 
such a DS would give a $(v,[1^{v-k},k],\lambda)$-HPDF.
The second question is whether, given positive integers $q$ and $n$, there exists a 
PDF whose related $K$ is $[q^0,q^1,q^2,q^3,\dots,q^{2n-1}]$.
The motivation is that such a PDF with $q=3$ would be Hadamard.

\section{PDFs and HPDFs with two blocks}

The following fact is extremely well-known (see, e.g., \cite{And}).
\begin{proposition}\label{complementary}
	If $B$ is a $(v,k,\lambda)$-DS in $G$, then the complement $\overline B$ of $B$ in $G$ is
	a $(v,v-k,v-2k+\lambda)$-DS.
\end{proposition}
The above explains why, as said in the introduction, a $(v,k,\lambda)$-DS and its complement give a $(v,[k,v-k],v-2k+2\lambda)$-PDF.

Although the following is a rather straightforward generalization of Proposition \ref{complementary},
we are not aware whether it was ever stated explicitly before.
\begin{proposition}\label{anderson}
	Let $B$ be a $k$-subset of an additive group $G$ of order $v$ and let $\overline B$ be the complement of $B$ in $G$.
	Let  $g$ be a non-zero element of $G$ and let $\lambda(g)$ be its multiplicity in $\Delta B$.
	Then the multiplicity $\overline{\lambda}(g)$ of $g$ in $\Delta\overline B$ is $v-2k+\lambda(g)$.
\end{proposition}
\begin{proof}
	Given subsets $X$ and $Y$ of $G$, set $$\Lambda_{X\times Y}(g)=\{(x,y)\in X\times Y \ | \ x-y=g\}.$$
	We have $|\Lambda_{B\times B}(g)|=\lambda$ and $|\Lambda_{\overline{B}\times \overline{B}}(g)|=\overline{\lambda}$
	by assumption.
	Let $(x_1,y_1)$, \dots, $(x_\lambda,y_\lambda)$ be the $\lambda$ pairs of $\Lambda_{B\times B}(g)$ and set
	$$B'=B\setminus\{x_1,\dots,x_\lambda\}\quad{\rm and}\quad B''=B\setminus\{y_1,\dots,y_\lambda\}.$$
	It is easy to see that 
	$$\Lambda_{{B}\times \overline{B}}(g)=\{(b,-g+b) \ | \ b\in B'\}\quad{\rm and}\quad \Lambda_{\overline{B}\times {B}}(g)=\{(g+b,b) \ | \ b\in B''\}$$
	so that $|\Lambda_{{B}\times \overline{B}}(g)|=|B'|=k-\lambda$ and $|\Lambda_{\overline{B}\times {B}}(g)|=|B''|=k-\lambda$.
	Finally, it is evident that  $\Lambda_{G\times G}(g)=\{(g+y,y) \ | \ y\in G\}$ so that $|\Lambda_{G\times G}(g)|=|G|=v$.
	Considering that the four sets $B\times B$,  $\overline{B}\times \overline{B}$,  ${B}\times \overline{B}$, and  $\overline{B}\times {B}$
	partition $G\times G$, we can write 
	$$|\Lambda_{{B}\times {B}}(g)|+|\Lambda_{\overline{B}\times \overline{B}}(g)|+|\Lambda_{{B}\times \overline{B}}(g)|+|\Lambda_{\overline{B}\times {B}}(g)|=|\Lambda_{G\times G}(g)|$$
	and hence, from what seen above, $\lambda+\overline{\lambda}+(k-\lambda)+(k-\lambda)=v$ which immediately gives the assertion.
\end{proof}

\begin{proposition}\label{trivialHPDFs}
	A PDF with only two blocks necessarily consists of a difference set and its complement.
	More specifically, a HPDF with only two blocks necessarily consists of a Hadamard difference set and its complement.
\end{proposition}
\begin{proof}
	If $\cal F$ is a $(v,[k,v-k],\lambda)$-PDF, then ${\cal F}=\{B,\overline B\}$ with $B$ a $k$-subset 
	of an additive group $G$ of order $v$ and $\overline B=G\setminus B$. 
	Given any element $g\in G\setminus\{0\}$
	denote by $\mu(g)$ and $\overline \mu(g)$ the multiplicities of $g$ in $\Delta B$ and $\Delta \overline B$, 
	respectively. By definition of a PDF we have $\mu(g)+\overline \mu(g)=\lambda$ for
	every $g\in G$. On the other hand we have $\overline \mu(g)=v-2k+\mu(g)$ by Proposition \ref{anderson}
	and then $v-2k+2\mu(g)=\lambda$, i.e., $\mu(g)={2k-v+\lambda\over2}$ is a constant.
	This means that $B$ is a difference set and the first assertion follows.
	
	Now assume that $\{B,\overline B\}$ is a $(v,[k,v-k],\lambda)$-HPDF so that we have $v=2\lambda$.
	Without loss of generality we can assume that $k\leq v-k$, i.e., $k\leq {v\over2}$.
	Then we deduce that $B$ is a $(v,k,\mu)$-DS with $\mu={2k-\lambda\over2}$. Considering that $v$ is even,
	$k-\mu$ must be a square by the Bruck-Ryser-Chowla theorem. It follows that ${\lambda\over2}=u^2$,
	hence $\mu=k-u^2$ for some integer $u$. Then, from the trivial identity $\mu(v-1)=k(k-1)$, we get 
	$(k-u^2)(4u^2-1)=k(k-1)$. Hence $k$ is a solution of the quadratic equation
	$x^2-4u^2x+4u^4-u^2=0$ so that either $k=2u^2+u$ or $k=2u^2-u$. On the other hand we have
	assumed $k\leq{v\over2}$, hence $k=2u^2-u$ and $\mu=u^2-u$.
	We conclude that $B$ is a $(4u^2,2u^2-u,u^2-u)$-DS, i.e., $B$ is a Hadamard difference set.
	The assertion follows.
\end{proof}

\section{Some necessary conditions}

The following proposition is straightforward.

\begin{proposition}\label{elementary}
	Let $\mathcal{F}$ be a $(G, [k_1, \ldots , k_t], \lambda)$-HPDF. Then
	\begin{enumerate}[(i)]
		\item $k_1 + \cdots + k_t = 2\lambda$;\label{prop1a}
		\item $k_1^2 + \cdots + k_t^2 = \lambda (2\lambda +1)$;\label{prop1b}
		\item $\lambda$ is even, hence $o(G) \equiv 0$ $\pmod 4$.\label{prop1c}
	\end{enumerate}
\end{proposition}
\begin{proof} Let $\mathcal{F}$ be a $(G, [k_1, \ldots , k_t], \lambda)$-HPDF.
	By definition of a HPDF, we have $k_1 + \cdots + k_t =|G|=2\lambda$.
	From $\Delta \mathcal{F} = \lambda \, (G \setminus \{0\})$
	we obtain $$|\Delta \mathcal{F}| = k_1(k_1-1) + \cdots + k_t(k_t-1) = \lambda (2\lambda-1)$$ 
	and then
	$$k_1^2 + \cdots k_t^2 = \lambda (2\lambda-1)+k_1 + \cdots + k_t.$$ 
	Then (ii) follows taking into consideration (i). 
	
	We obviously have $k_1 + \cdots + k_t\equiv k_1^2 + \cdots + k_t^2$ (mod 2).
	Thus, using (i) and (ii), we get $2\lambda\equiv \lambda(2\lambda+1)$ (mod 2)
	and (iii) follows.
\end{proof}
As a consequence of the above proposition, every HPDF in a group $G$ has $o(G)$ necessarily doubly even.
However, it is worth to observe that this can be deduced in another way as follows. Assuming that $\cal F$ is a HPDF
in $G$ and $o(G)\equiv2$ (mod 4), then $\lambda={o(G)\over2}$ would be odd and this is absurd since the multiplicity
in $\Delta{\cal F}$ of any involution $i$ of $G$  is necessarily even: indeed, if $x-y=i$ is a representation of $i$ as a
difference from $\cal F$, then $y-x=i$ is a representation of $i$ as a
difference from $\cal F$ as well.

In the following table we list the very first parameter sets satisfying the necessary conditions of Proposition \ref{elementary}
disregarding those were $K$ has size 2 in view of Proposition \ref{trivialHPDFs}.
$$
\begin{array}{cc}
\begin{array}{c|c|c}
v & K & \lambda \\
\hline
{\color{red}{20}} &  {\color{red}{[ 1,2,3,14  ]}} & {\color{red}10} \\
{\color{teal} 24} &  {\color{teal} [ 1^3,2^2,17  ]} & {\color{teal} 12} \\
\color{red}28 &  {\color{red}{[ 1,9,18 ]}} & \color{red}14 \\
\color{red}28  &  {\color{red}{[ 3,6,19  ]}} & \color{red}14 \\
{\color{blue} 32 } &  {\color{blue} [ 2^2,6,22  ] } & {\color{blue} 16}\\
{\color{teal} 36} &  {\color{teal} [ 3,9,24  ]} & {\color{teal} 18}\\
\color{red}36  &  {\color{red}{[ 3,4^2,25  ]}} & \color{red}18 \\
\color{red}36  &  {\color{red}{[ 1^5,6,25  ]}} & \color{red}18 \\
{\color{teal} 40} & {\color{teal}  [ 1,3,9,27  ]} & {\color{teal} 20} \\
40  &  [ 3^4, 28 ] & 20 \\
40 &   [ 1^2,3^2,4,28 ] & 20 \\
\color{red}40 &   {\color{red}{[ 1^4,4^2,28 ]}} & \color{red}20 \\ 
40 &  [ 1^3,2^2,5,28 ] & 20 \\
\end{array}
\end{array}
$$

An HPDF with ``blue parameter set" has been found by Buratti \cite{Bur19} and it
is the only non-elementary HPDF known at this moment.

In Section \ref{new} we will determine an HPDF with ``green parameter set" in several groups.

By exhaustive computer search we have checked that an HPDF with ``red parameter set" does not exist.

The existence of an HPDF with ``uncolored parameter set" is still in doubt. 

The HPDFs with only two blocks have been already characterized in Proposition \ref{trivialHPDFs}.
Let us see what we can say about HPDFs with three blocks.
\begin{proposition}\label{k1k2k3}
	In a $(v,[k_1,k_2,k_3],\lambda)$-HPDF 
	we necessarily have
	$$k_{1,2}={2\lambda-k_3\pm\sqrt{2\lambda(2k_3+1)-3k_3^2}\over2}$$
\end{proposition}
\begin{proof}
	By Proposition \ref{elementary} we have $k_1+k_2+k_3=2\lambda$ and $k_1^2+k_2^2+k_3^2=\lambda(2\lambda+1)$.
	The second identity can be rewritten as $(k_1+k_2)^2-2k_1k_2+k_3^2=2\lambda^2+\lambda$ and then, in view of the
	first identity, we have $(2\lambda-k_3)^2-2k_1k_2+k_3^2=2\lambda^2+\lambda$ which gives
	$$k_1\cdot k_2={2\lambda^2-(4k_3+1)\lambda+2k_3^2\over2}$$ 
	Considering that $k_1+k_2=2\lambda-k_3$ and recalling that two numbers having sum $s$ and product $p$ are the solutions of the quadratic
	equation $x^2-sx+p=0$, after trivial computations we get the assertion.
\end{proof}

The possible block-sizes of a $(v,[k_1,k_2,k_3],\lambda)$-HPDF
are strongly limited by the above proposition. Indeed we have the following.

\begin{corollary}\label{forbidden}
	The existence of a $(v,[k_1,k_2,k_3],\lambda)$-HPDF necessarily implies that no prime divisor of
	$(2k_1+1)(2k_2+1)(2k_3+1)$ is congruent to $5$ $($mod $6)$.
\end{corollary}
\begin{proof}
	Assume that a $(v,[k_1,k_2,k_3],\lambda)$-HPDF exists and let $p$
	be a prime factor of $(2k_1+1)(2k_2+1)(2k_3+1)$. Up to a reordering of the indices we can assume that $p$ is a divisor of $2k_3+1$
	and hence $p$ cannot divide $k_3$.
	By Proposition \ref{k1k2k3} it is clear that $2\lambda(2k_3+1)-3k_3^2$ must be a perfect square. So, in particular, it must be
	a square modulo $p$. Considering that $p$ is a divisor of $2k_3+1$ we have that $2k_3+1\equiv0$ (mod $p$). Also,
	$k_3^2$ is a non-zero square of $\Z_p$ since $p$ does not divide $k_3$. We conclude that $-3$ must be a square of $\Z_p$ and then,
	by the Quadratic Law of Reciprocity, we have either $p=3$ or $p\equiv1$ (mod 6).
\end{proof}

As a consequence of Corollary \ref{forbidden}, in a $(v,[k_1,k_2,k_3],\lambda)$-HPDF we cannot have, for instance, blocks
of size 2, 5, 7, 8, 11, 12, 14, 16, 17, ...

\section{Exploiting partial difference sets}

Note that Corollary \ref{forbidden} does not forbid the existence of a 
$$
(v,[k_1,k_2,k_3],\lambda)\mbox{-HPDF}
$$
 with a block of size 1 and indeed,
using Proposition \ref{k1k2k3}, we can see, for instance, that for every $n\geq2$ 
$$\displaystyle\biggl{(}{3^{2n-1}+1\over2}, \ \biggl{[}{3^{2n-1}+3^n\over2}, \ {3^{2n-1}-3^n\over2}, \ 1\biggl{]}, \ {3^{2n-1}+1\over4}\biggl{)}$$
is an admissible parameter set for a HPDF.
On the other hand we are going to see that a $(v,[k_1,k_2,1],\lambda)$-HPDF cannot exist.
To prove this we have to exploit a result on {\it partial difference sets}.

A $(v,k,\alpha,\beta)$ partial difference set (PDS) in a group $G$ is a $k$-subset $B$ of $G$ such that $\Delta B$
covers $\alpha$ times every non-zero element of $B$ and $\beta$ times every non-zero element of $G\setminus B$.

Among the various necessary conditions for the existence of
a non-trivial PDS we will need the following (see Proposition 3(d) in \cite{Ma}).
\begin{lemma}\label{Ma}
	If there exists a $(v,k,\alpha,\beta)$-PDS
	and $\gamma:=(\alpha-\beta)^2+4(k-\beta)$ is not a square, then
	we have $(v,k,\alpha,\beta)=(4t+1,2t,t-1,t)$ for a suitable integer $t$.
\end{lemma}

In the next proposition we will see that if a PDF in a group $G$ has exactly three blocks one of which has size 1, then 
each of the other two blocks is a PDS or a DS in $G$.

\begin{proposition}\label{partial}
	Let ${\cal F}=\{B_1,B_2,\{0\}\}$ be a $(v,[k_1,k_2,1],\lambda)$-PDF in $G$. Then, for $i=1,2$ we have:
	\begin{enumerate}[(i)]
		\item $B_i$ is a $(v,{v-1\over2},{v-3\over4})$-DS if $v-\lambda$ is odd;
		\item $B_i$ is a $(G,k_i,\alpha_i,\alpha_i+1)$-PDS with $\alpha_i=k_i+{\lambda-v\over2}$ if $v-\lambda$ is even.
	\end{enumerate}
\end{proposition}
\begin{proof}
	For $i=1,2$ and for $g\in G$, let $\lambda_i(g)$ be the multiplicity of $g$ in $\Delta B_i$.
	The assumption that $\cal F$ is a $(v,[k_1,k_2,1],\lambda)$-PDF means that we have 
	\begin{equation}\label{lambda}
	\lambda_1(g)+\lambda_2(g)=\lambda\quad \forall \ g\in G\setminus\{0\}
	\end{equation}
	Let $\overline B_1$ be the complement of $B_1$ in $G$ and let $\overline{\lambda_1}(g)$ be the multiplicity of $g$ in $\Delta\overline{B_1}$.
	We have $\overline{B_1}=B_2\cup\{0\}$ so that the possible representations of $g\in G$ as a difference from $\overline{B_1}$ which are
	not representations of $g$ as a difference from $B_2$ 
	are $g=g-0$ if $g\in B_2$, and $g=0-(-g)$ if $-g\in B_2$. Thus we can write:
	\begin{equation}\label{cases}
	\lambda_2(g)=
	\begin{cases}
	\overline{\lambda_1}(g) & {\rm if} \ \{g,-g\}\subset B_1;\cr
	\overline{\lambda_1}(g)-2 & {\rm if} \ \{g,-g\}\subset B_2; \cr 
	\overline{\lambda_1}(g)-1 & {\rm if} \ |\{g,-g\}\cap B_i|=1 \ {\rm for} \ i=1,2\cr 
	\end{cases}
	\end{equation}
	In the following, we distinguish two cases according to the parity of $v-\lambda$
	and keep in mind that we have 
	\begin{equation}\label{anderson2}\overline{\lambda_1}(g)=v-2k_1+\lambda_1(g)
	\end{equation} 
	in view of Proposition \ref{anderson}.
	
	\medskip
	\textbf{Case 1:} $v-\lambda$ is odd.
	
	{\leftskip=0.5cm
	Observe that for any given $g\in G\setminus\{0\}$ we cannot have $\{g,-g\}\subset B_1$ otherwise (\ref{lambda}), (\ref{cases}) and (\ref{anderson2}) would give
	$\lambda=v-2k_1+2\lambda_1(g)$ contradicting the assumption that $v-\lambda$ is odd.
	
	Analogously, we cannot have $\{g,-g\}\subset B_2$ otherwise (\ref{lambda}), (\ref{cases}) and (\ref{anderson2}) would give
	$\lambda=v-2k_1+2\lambda_1(g)-2$ contradicting again that $v-\lambda$ is odd.
	
	We conclude that $g$ and $-g$  lie in different $B_i$'s for every $g\in G\setminus\{0\}$. 
	It follows, in particular, that $G$ is involution-free:
	if $g$ is an involution, we cannot have $g\in B_1$ and $-g\in B_2$ since $g$ and $-g$ are the same.
	Thus $G$ has odd order and we 
	have $B_2=-B_1$ so that $k_1=k_2={v-1\over2}$. 
	Considering that ${\cal F}$ is a $(v,[k_1,k_2,1],\lambda)$-PDF we must have 
	$\lambda(v-1)=k_1(k_1-1)+k_2(k_2-1)$ and hence $\lambda(v-1)=2\cdot{v-1\over2}\cdot{v-3\over2}$
	which gives $\lambda={v-3\over2}$.
	Finally, (\ref{cases}) and (\ref{anderson2}) give 
	$\lambda_2(g)=v-2\cdot{v-1\over2}+\lambda_1(g)-1=\lambda_1(g)$. It follows, by (\ref{lambda}), that $\lambda_1(g)=\lambda_2(g)={\lambda\over2}={v-3\over4}$.
	It is now clear that $B_1$ and $B_2$ are $(v,{v-1\over2},{v-3\over4})$-DSs.

}
	\medskip
	\textbf{Case 2:} $v-\lambda$ is even.
	
{\leftskip=0.5cm
	Here, there is no $g\in G$ such that $g\in B_1$ and $-g\in B_2$.
	Indeed, in the opposite case,  (\ref{lambda}), (\ref{cases}) and (\ref{anderson2}) would give
	$\lambda=v-2k_1+2\lambda_1(g)-1$ contradicting the assumption that $v-\lambda$ is even.
	Thus, by (\ref{cases}), we have
	$$\lambda_2(g)=
	\begin{cases}
	v-2k_1+\lambda_1(g) & {\rm if} \ g\in B_1;\cr
	v-2k_1+\lambda_1(g)-2 & {\rm if} \ g\in B_2
	\end{cases}$$
	Using again (\ref{lambda}) and (\ref{anderson2}) we get $\lambda=v-2k_1+2\lambda_1(g)$ or $v-2k_1+2\lambda_1(g)-2$ according to
	whether $g\in B_1$ or $g\in B_2$, respectively. Solving these identities with respect to $\lambda_1(g)$ we finally get
	$$\lambda_1(g)=
	\begin{cases}
	\alpha_1 & {\rm if} \ g\in B_1;\cr
	\alpha_1+1 & {\rm if} \ g\in B_2
	\end{cases}$$
	This precisely means that $B_1$ is a $(G,k_1,\alpha_1,\alpha_1+1)$-PDS.
	In the same way, exchanging the roles of $B_1$ and $B_2$ one gets that 
	$B_2$ is a $(G,k_2,\alpha_2,\alpha_2+1)$-PDS.

}

\end{proof}

\begin{corollary}
	A $(v, [k_1,k_2,1], \lambda)$-HPDF cannot exist.
\end{corollary}
\begin{proof}
	Assume that there exists a $(v, [k_1,k_2,1], \lambda)$-HPDF so that $v-\lambda=\lambda$ is even.
	Up to a translation we can assume that the block of size 1 is $\{0\}$.
	Then, by Proposition \ref{partial}, the block of size $k_1$ is a $(2\lambda,k_1,\alpha,\alpha+1)$-PDS with
	$\alpha=k_1-{\lambda\over2}$.  Here the parameter $\gamma$ mentioned in Lemma \ref{Ma} is 
	$1+4(k_1-k_1+{\lambda\over2})=2\lambda+1$. Thus, in view of the same lemma, $2\lambda+1$ is a perfect square, say $2\lambda+1=\mu^2$,
	otherwise we should have $v=4t+1$ which is absurd.
	By Proposition \ref{k1k2k3} it is also necessary that $6\lambda-3$ is a perfect square. Now note that we have $6\lambda-3=3(2\lambda-1)=3(\mu^2-2)$
	so that $3$ should divide $\mu^2-2$. This is absurd since 2 is not a square modulo 3.
\end{proof}

\section{Exploiting subgroups of index 2}

The following proposition exploits the possible existence of a subgroup of index 2.

\begin{proposition}\label{index2} 
	Let ${\cal F}=\{B_1,\dots,B_t\}$ be a $(G, [k_1, \ldots, k_t], \lambda)$-HPDF, assume that $G$ has a  subgroup $H$ of index $2$,
	and set $|B_i\cap H|=s_i$ for $i=1,\dots,t$. Then the following identities hold:
	$$s_1+ \ ... \ + s_t=\lambda\quad\quad{\rm and}\quad\quad
	2s_1(k_1-s_1)+ \ ... \ +2s_t(k_t-s_t)=\lambda^2$$
\end{proposition}
\begin{proof}
	For $i=1,\dots,t$, set $B'_i=B_i\cap H$ and $B''_i=B_i\setminus H$ so that $|B'_i|=s_i$ and $|B''_i|=k_i-s_i$.
	The first identity follows from the fact that the $B_i$'s partition $H$ which has order $\lambda$.
	Now note that we have
	$$\Delta B_i=\Delta B'_i   \ \cup \ \Delta B''_i \ \cup \ (B'_i - B''_i) \ \cup \ (B''_i-B_i).$$ 
	Also note that $\Delta B'_i \ \cup \ \Delta B''_i$ is a multisubset of $H$ and that both $B'_i - B''_i$ and $B''_i - B'_i$ are multisubsets of $G\setminus H$ 
	of size $|B'_i|\cdot|B''_i|=s_i(k_i-s_i)$.
	Thus we can say that $\Delta B_i$ has exactly $2s_i(k_i-s_i)$ elements in $G\setminus H$. 
	Then, considering that $\Delta{\cal F}$ covers every element of $G\setminus H$ exactly $\lambda$ times, we conclude that we have
	$\sum_{i=1}^t2s_i (k_i - s_i) =\lambda\cdot|G\setminus H|=\lambda^2$, i.e., the second identity holds.
\end{proof}
As a consequence we have the following.
\begin{corollary}\label{Diophantus} 
	If there exists a $(G, [k_1, \ldots, k_t], \lambda)$-HPDF and $G$ has a  subgroup of index $2$, then
	the diophantine system
	$$
	\left \{ \begin{array}{rcrcl}
	x_1&+ \ ... \ +& x_t&=&\lambda\\
	2x_1(k_1-x_1)&+ \ ... \ +&2x_t(k_t-x_t)&=&\lambda^2
	\end{array}
	\right.
	$$
	has a solution $(s_1,\dots,s_t)$ with $0\leq s_i\leq k_i$ for each $i$.
\end{corollary}

As application of the above corollary one can see that none of these $K$, though admissible, can be the multiset of block-sizes of a HPDF:
$$ [ 50, 20, 5, 1 ];\quad [ 52, 23, 2, 1, 1, 1 ];\quad [ 73, 38, 3, 2 ];$$
$$\quad [ 77, 28, 8, 3 ];\quad [ 79, 31, 7, 3 ];\quad [ 81, 21, 16, 1, 1 ];\quad [ 104, 35, 14, 3 ].$$

As a matter of fact we have to admit that the admissible $K$ which are ruled out by Corollary \ref{Diophantus}
do not appear so many. On the other hand Proposition \ref{index2} has been useful to limit our 
computer search for HPDFs of small orders.

\section{New HPDFs}\label{new}

In this section we present HPDFs with the three new parameter sets
$$( 24, [1^3,2^2,17], 12),\quad\quad (36,  [3,9, 24], 18)\quad {\rm and}\quad (40, [1, 3, 9, 27], 20).$$
We will use, in particular, some dihedral groups and dicyclic groups of small orders.
We recall that the dihedral group of order $2n$, denoted $D_{2n}$, is the group with defining relations
$$\langle x, y \ | \ x^{n}=1; \ y^2=1; \ yx^i=x^{-i}y\rangle$$

We also recall that the dicyclic group of order $4n$, denoted $Q_{4n}$, is the group with defining relations
$$\langle x, y \ | \ x^{2n}=1; \ y^2=x^n; \ yx^i=x^{-i}y\rangle$$

\subsection{$( 24, [1^3,2^2,17], 12)$-HPDFs}\label{24}
We found an example of a $( 24, [1^3,2^2,17], 12)$-HPDF in each of the following groups: $C_3 \rtimes C_8$, $SL(2,3)$ and $C_3 \times D_8$.

\noindent
\textbf{$G=C_3 \rtimes C_8$
}

{\leftskip=0.5cm

This is the semidirect product of $C_3$ by $C_8$ with defining relations
$$
C_3 \rtimes C_8 = \langle a, b \, | \,   a^8 = b^3 = 1, ab^{-1} = ba \rangle
$$
Thus the elements of $G$ are of the form $a^ib^j$ with $0\leq i\leq 7$ and $0\leq j\leq 2$.
The difference (even though we should say ``ratio" since we are in multiplicative notation) 
between two elements $a^{i_1}b^{j_1}$ and $a^{i_2}b^{j_2}$ is given by
\begin{equation}\label{semidirectrule}
(a^{i_1}b^{j_1})(a^{i_2}b^{j_2})^{-1}=a^{i_1-i_2}b^{(-1)^{i_2}(j_1-j_2)}
\end{equation}
Let ${\cal F}=\{B_1,B_2,B_3,B_4,B_5,B_6\}$ be the partition of $G$ defined as follows:
$$ B_1=\{ 1, a, a^2, a^3, a^4, a^6, a^7, b, ab, a^3b, a^4b, a^5b, a^6b, b^2, ab^2, a^2b^2, a^4b^2 \};$$
$$B_2 =\{ a^3b^2\};\quad B_3 = \{ a^5b^2\};\quad B_4 = \{ a^7b^2 \};$$
$$B_5 =\{ a^5, \ a^2b \};\quad B_6 =\{ a^7b, \ a^6b^2 \}.$$
Using (\ref{semidirectrule}) it is straightforward to check that $\cal F$ is a $(G,[1^3,2^2,17],12)$-HPDF.

}

\noindent
\textbf{$G=SL(2,3)$}

{\leftskip=0.5cm
This is the 2-dimensional special linear group over $\Z_3$. Its elements are the $2\times2$ matrices
with elements in $\Z_3$ and determinant equal to 1. 
Let ${\cal F}=\{B_1,B_2,B_3,B_4,B_5,B_6\}$ be the partition of $G$ defined as follows:
$$B_1=\biggl{\{}\begin{pmatrix}2&1\cr 0&2\end{pmatrix}\biggl{\}};\quad\quad
B_2=\biggl{\{}\begin{pmatrix}1&2\cr 0&1\end{pmatrix}\biggl{\}};\quad\quad B_3=\biggl{\{}\begin{pmatrix}0&2\cr1&2\end{pmatrix}\biggl{\}};$$
$$B_4=\biggl{\{}\begin{pmatrix}0&2\cr1&0\end{pmatrix}, \ \begin{pmatrix}1&1\cr 1&2\end{pmatrix}\biggl{\}};\quad\quad\quad
B_5=\biggl{\{}\begin{pmatrix}2&1\cr2&0\end{pmatrix}, \ \begin{pmatrix}2&2\cr 0&2\end{pmatrix}\biggl{\}};$$
$$B_6=G \ \setminus \ (B_1 \ \cup \ B_2 \ \cup \ B_3 \ \cup \ B_4 \ \cup \ B_5).$$

It is straightforward to check that $\cal F$ is a $(G,[1^3,2^2,17],12)$-HPDF.

}

\noindent
\textbf{$G=\Z_3 \times D_8$}

{\leftskip=0.5cm
The reader can easily recognize that the partition of $G$ into the blocks listed below is a 
$(G,[1^3,2^2,17],12)$-HPDF.
$$B_1=\{(0,x^2)\};\quad B_2=\{(2,xy)\};\quad B_3=\{(2,x^3y)\};$$
$$B_4=\{(1,x^3), \ (2,x^3)\};\quad\quad B_5=\{(1,y), \ (2,x^2y)\};$$
$$B_6=G \ \setminus \ (B_1 \ \cup \ B_2 \ \cup \ B_3 \ \cup \ B_4 \ \cup \ B_5).$$

}

\subsection{$(36,  [3,9,24], 18)$-HPDFs}\label{36}
We found an example ${\cal F}=\{A,B,C\}$ of a $(36,  [3,9,24], 18)$-HPDF in the groups
$$\Z_6 \times \Z_6,\quad\quad \Z_{3} \times \Z_{12},\quad\quad \Z_3 \times \mbox{Q}_{12},\quad\quad D_6 \times D_6, \quad\quad\Z_6 \times D_6,$$ 
$$\Z_3 \times A_4,\quad\quad\quad\Z_3 \rtimes \mbox{Q}_{12},\quad\quad\quad\Z_3^2 \rtimes \Z_4,\quad\quad\quad\Z_2 \times \Z_3 \rtimes D_6.$$
We present our example for each of the first five.

\noindent
\textbf{$G=\Z_6 \times \Z_6$
}

\smallskip
{\leftskip=0.5cm
$A=\{(1,1),(1,3),(1,5)\};$

\smallskip
$B=\{\{0,2),(0,3),(1,4),(2,0),(2,5),(3,4),(4,1),(4,4),(5,4)\};$

\smallskip
$C=G\setminus(A \ \cup \ B).$

}

\noindent
\textbf{$G=\Z_{3} \times \Z_{12}$
}

{\leftskip=0.5cm
\smallskip
$A=\{(1,1),(1,5),(1,9)\};$

\smallskip
$B=\{\{0,2),(0,3),(0,4),(1,2),(1,8),(1,11),(2,0),(2,2),(2,7)\};$

\smallskip
$C=G\setminus(A \ \cup \ B).$

}

\smallskip

\noindent
\textbf{$G=\Z_3 \times \mbox{Q}_{12}$}

{\leftskip=0.5cm
\smallskip
$A=\{(0,xy),(1,xy),(2,xy)\};$

\smallskip
$B=\{\{1,1),(0,x^3),(0,x^2),(2,y),(1,x^5),(2,x^4y),(2,x^4),(2,x^2y),(2,x)\};$

\smallskip
$C=G\setminus(A \ \cup \ B).$

}

\noindent
\textbf{$G=D_6 \times D_6$}

{\leftskip=0.5cm
\smallskip
$A=\{(y,xy),(xy,xy),(x^2y,xy)\};$

\smallskip
$B=\{(1,x^2y),(x,y),(x^2,1),(x^2,x),(x^2,x^2),(x^2,xy),(y,1),(xy,x^2),(x^2y,x)\};$

\smallskip
$C=G\setminus(A \ \cup \ B).$

}

\smallskip
\noindent \textbf{$G=\Z_6 \times D_6$}

{\leftskip=0.5cm
\smallskip
$A=\{(1,xy),(3,xy),(5,xy)\};$

\smallskip
$B=\{\{0,x),(1,x^2),(2,1),(3,1),(4,x^2),(4,y),(4,xy),(4,x^2y),(5,x)\};$

\smallskip
$C=G\setminus(A \ \cup \ B).$

}

\subsection{A cyclic $(40, [1, 3, 9, 27], 20)$-HPDF}\label{40}

This is our unique example of a HPDF in a cyclic group. Also, it has maximum order among the few non-elementary HPDFs 
that are known at this moment. Thus it is surprising that this is also the unique example that we have been able to get by hand
without any use of the computer. The idea was the following. Start from any cyclic $(40,13,4)$-DS, that is a {\it Singer} difference set. 
Such a difference set is available in the literature. One is, for instance, the following (see \cite{JPS06}, page 427):
$$D=\{1,2,3,5,6,9,14,15,18,20,25,27,35\}$$
By Corollary \ref{complementary}, $\overline{D}:=\Z_{40}\setminus D$ is a $(40,27,18)$-DS. Then, if we are able to partition $D$ into three subsets $A$, $B$, $C$ of sizes $1$, $3$, $9$
in such a way that $\Delta A \ \cup \ \Delta B \ \cup \ \Delta C$ is all $\Z_{40}\setminus\{0\}$ twice, it is obvious that $${\cal F}=\{A,B,C,\overline{D}\}$$ would be a
$(40, [1,3,9,27], 20)$-HPDF. Well, the desired partition of $D$ has $A$, $B$, $C$ as follows:
$$A=\{13\};\quad B=\{5,15,25\};\quad C=\{1,2,3,6,9,14,18,20,27\}.$$
This is readily seen from the {\it difference tables} of $B$ and $C$ below (of course $\Delta A$ is empty).

\small
\medskip
\begin{center}
	\begin{tabular}{|l|c|r|c|r|c|r|c|r|c|r|c|r|}
		\hline {$$} & 5 & 15 & 25    \\
		\hline
		\hline $5$ & $-$ & \bf30 & \bf20  \\
		\hline $15$ & $\bf10$ & $-$ & \bf30  \\
		\hline $25$ & $\bf20$ & \bf10 & $-$  \\
		\hline
		\end{tabular}
		
		\medskip
		
		\begin{tabular}{|l|c|r|c|r|c|r|c|r|c|r|c|r|}
		\hline {$$} & 1 & 2 & 3 & 6 & 9 & 14 & 18 & 20 & 27   \\
		\hline
		\hline $1$ & $-$ & \bf39 & \bf38 & \bf35 & \bf32 & \bf27 & \bf23 &\bf21 & \bf14  \\
		\hline $2$ & \bf1 & $-$ & \bf39 & \bf36 & \bf33 & \bf28 & \bf24 &\bf22 & \bf15  \\
		\hline $3$ & \bf2 & \bf1 & $-$ & \bf37 & \bf34 & \bf29 & \bf25 &\bf23 & \bf16  \\
		\hline $6$ & \bf5 & \bf4 & \bf3 & $-$ & \bf37 & \bf32 & \bf28 &\bf26 & \bf19  \\
		\hline $9$ & $\bf8$ & \bf7 & \bf6 & \bf3 & $-$ & \bf35 & \bf31 &\bf29 & \bf22  \\
		\hline $14$ & $\bf13$ & \bf12 & \bf11 & \bf8 & \bf5 & $-$ & \bf36 &\bf34 & \bf27  \\
		\hline $18$ & $\bf17$ & \bf16 & \bf15 & \bf12 & \bf9 & \bf4 & $-$ &\bf38 & \bf31  \\
		\hline $20$ & $\bf19$ & \bf18 & \bf17 & \bf14 & \bf11 & \bf6 & \bf2 & $-$ & \bf33  \\
		\hline $27$ & $\bf26$ & \bf25 & \bf24 & \bf21 & \bf18 & \bf13 & \bf9 &\bf7 & $-$  \\
		\hline
	\end{tabular}\quad\quad\quad
\end{center}

\normalsize
It is worth observing that $C$ is a $(10,4,9,2)$ {\it relative difference set}. It means that
$\Delta C$ is precisely twice $\Z_{10\cdot4}\setminus N$ where $N$ is the subgroup of $\Z_{10\cdot4}$
of order 4 (see \cite{JPS06}).

\section{New infinite families of PDFs}
Most composition constructions for PDFs make use of {\it difference matrices} 
and lead to PDFs whose block sizes belong, almost all, to the set of block sizes of the
component PDFs (see \cite{BYW, LWG}).
Buratti \cite{Bur19} motivated the introduction of HPDFs showing that any single $(v,K,\lambda)$-HPDF 
is the {\it ancestor} of an infinite series of PDFs where, apart from one special block of size $2\lambda$, 
the size of every other block is the double of some $k\in K$. We are going to recall the main application 
of his construction and then we determine the {\it descendants} of our new examples.

Following \cite{BBGRT}, the maximal prime power factors of a given integer $v$ will be called
{\it components} of $v$, and $\F_v$ will denote the ring which is the direct product of the fields 
whose orders are the components of $v$. Thus, for instance, $\F_{63}=\F_9\times\F_7$.

\begin{theorem}\cite{Bur19}
	If there exists a $(G, [k_1, \ldots , k_t], \lambda)$-HPDF and all the components of $2n+1$ are greater than $2\cdot\max\{k_1,\dots,k_t\}$, 
	then there exists a $(2\lambda(2n + 1), [(2k_1)^n, \ldots,(2k_t)^n, 2\lambda], 2\lambda)$-PDF in $G\times \F_{2n+1}$.
\end{theorem}

Applying the above theorem using the new examples of HPDFs obtained in Section \ref{new} we obtain the following results.

\begin{corollary}
	If all the components of $2n+1$ are greater than $34$, then there exists a 
	$(48n + 24, [34^n,4^{2n},2^{3n},24], 24)$-PDF in $G\times\F_{2n+1}$ for each of the three
	groups $G$ considered in \ref{24}.
\end{corollary}

The first $n$ for which the above corollary can be applied is 18.
In this way one gets a $(984, [34^{18},4^{36},2^{54},24], 24)$-PDF in $G\times \F_{37}$.

\begin{corollary}
	If all the components of $2n+1$ are greater than $48$, then there exists a 
	$(72n + 36, [6^n,18^n,48^n,36], 36)$-PDF  in $G\times\F_{2n+1}$ for each of the nine
	groups $G$ considered in \ref{36}.
\end{corollary}
The first $n$  for which the above corollary can be applied is 24. 
In this way one gets a $(1764,$ $[6^{24},18^{24},48^{24},36], 36], 36)$-PDF. 
\begin{corollary}
	If all the components of $2n+1$ are greater than $54$, 
	then there exists a $(80n + 40, [2^n,6^n,18^n,54^n,40], 40)$-PDF in $\Z_{40}\times \F_{2n+1}$.
\end{corollary}

The first $n$  for which the above corollary can be applied is 29.  
In this way one gets a $(2360, [2^{29},6^{29},18^{29},54^{29},40], 40)$-PDF.

\section{A pair of open questions}

It is obvious that any $(v,k,\lambda)$ difference set $B$ with $v=2\lambda$
gives rise to a $(v,[1^{v-k},k],\lambda)$-HPDF consisting of $B$ 
and all possible singletons $\{g\}$ with $g\in G \setminus B$.
Thus, for our purposes, it is worth to look for difference sets with this property.
Let us call them {\it Pell difference sets}. 
The existence of a Pell $(v,k,\lambda)$-DS obviously requires that $k(k-1)=\lambda(2\lambda-1)$
and, by the Bruck-Ryser-Chowla theorem, that $k-\lambda$ is a square. On the other hand 
the latter condition is redundant. Indeed, as shown in the following, the first condition implies the second.
\begin{proposition}
	If $k$ and $\lambda$ are integers such that $k(k-1)=\lambda(2\lambda-1)$, then $k-\lambda$ is a square.
\end{proposition}
\begin{proof}
	Set $d=\gcd(k,\lambda)$. Thus we have $k=de$ and $\lambda=df$ with $e$ and $f$ coprime integers.
	From the given equality we get $e(de-1)=f(2df-1)$ and then, considering that $\gcd(e,f)=1$, we necessarily 
	have $de-1=fg$ and $2df-1=eg$ for some integer $g$.
	Subtracting the second equality from the first one we get $d(e-2f)=g(f-e)$. We clearly have $\gcd(d,g)=1$ and hence 
	$d=f-e$ and $g=e-2f$. We conclude that $k-\lambda=d(e-f)=d^2$.
\end{proof}
Apart from the trivial $(4,3,2)$-DS, no other Pell difference set is known.
\begin{question}\label{Pell}
	Does there exist a Pell difference set of order $v>4$?
\end{question}
Here are the first possible triples $(v,k,\lambda)$ for 
which a non-trivial Pell $(v,k,\lambda)$ difference set may exist:

\smallskip
$\begin{array}{l}
(120, \ 85, \ 60) \\
(4060, \ 2871, \ 2030) \\
(137904, \ 97513, \ 68952) \\
(4684660, \ 3312555, \ 2342330) \\
(159140520, \ 112529341, \ 79570260)\\
\end{array}
$

Let us consider the first one. A putative $(120,85,60)$ is the complement of a $(120,35,10)$-DS. Even though several authors
investigated the possible existence of a DS with these parameters, they have been only able to rule out some groups (see \cite{Becker}).

In view of these consideration and the fact that the admissible values of $v$ grow up very rapidly, we are afraid that Question \ref{Pell} is really hard.

\bigskip
We come now to another question which is probably easier.

The parameter set of the last HPDF constructed in Section \ref{new}
can be written as 
$$
\left({3^4-1\over2},[3^0,3^1,3^2,3^3],{3^4-1\over4}\right).
$$
Inspired by this, we have noticed that 
$$
\left({q^{2n}-1\over q-1},[q^0,q^1,q^2,q^3,\dots,q^{2n-1}],{q^{2n}-1\over q+1}\right)
$$
is an admissible parameter set of a PDF for every positive integer $q$ (not necessarily a prime power!).
Indeed, if we set $v={q^{2n}-1\over q-1}$, $k_{i+1}=q^i$ for $0\leq i\leq 2n-1$, and $\lambda={q^{2n}-1\over q+1}$,
we see that we have:
\begin{itemize}
	\item[1)] 
	$\displaystyle k_1+k_2+\dots+k_{2n}=\sum_{i=0}^{2n-1}q^i={q^{2n}-1\over q-1}=v$;
	
	\medskip
	\item[2)]  
	$\displaystyle k_1(k_1-1)+k_2(k_2-1)+\dots+k_{2n}(k_{2n}-1)=$\\ \\
	$\begin{array}{l}
	= (k_1^2+k_2^2+\dots+k_{2n}^2)-(k_1+k_2+\dots+k_{2n})= \\ \\
	= \displaystyle \sum_{i=0}^{2n-1}q^{2i}-{q^{2n}-1\over q-1}={q^{4n}-1\over q^2-1}-{q^{2n}-1\over q-1}= \\ \\
	= \displaystyle{(q^{2n}-1)(q^{2n-1}-q)\over q^2-1}=\lambda(v-1).\\
	\end{array}
	$
\end{itemize}

\medskip
Thus the following question naturally arises.
\begin{question}
	Given positive integers $q$ and $n$, does there exist a PDF whose $K$ is $[q^0,q^1,q^2,q^3,\dots,q^{2n-1}]$?
\end{question}

For now, we know that the answer is positive for $q=3$ and $n=1,2$. 
A positive answer for $q=3$ and any $n$ would give the first infinite family of HPDFs.

\section*{Acknowledgements}
The author is supported in part by the Croatian Science Foundation under the project 9752.

\end{document}